\documentclass[english]{gretsi}
\usepackage{cite}
\usepackage{amsmath,amssymb,amsfonts,amsthm}
\usepackage{algorithmic}
\usepackage{graphicx}
\usepackage{textcomp}
\usepackage{xcolor}
\def\BibTeX{{\rm B\kern-.05em{\sc i\kern-.025em b}\kern-.08em
    T\kern-.1667em\lower.7ex\hbox{E}\kern-.125emX}}

\newtheorem{Def}{Definition}[section]

\newtheorem{Prop}[Def]{Proposition}

\begin{document}

\title{Weighted tensorized fractional Brownian textures}

\auteurs{
  \auteur{C\'eline}{Esser}{celine.esser@uliege.be}{1}
  \auteur{Claire}{Launay}{claire.launay@univ.ubs.fr}{2}
  \auteur{Laurent}{Loosveldt}{l.loosveldt@uliege.be}{1}
  \auteur{B\'eatrice}{Vedel}{beatrice.vedel@univ-ubs.fr}{2}
}

\affils{
  \affil{1}{D\'epartement de Math\'ematique,
        Université de Liège, Belgique
  }
  \affil{2}{UMR CNRS 6205, LMBA,
        Université de Bretagne Sud,
        Vannes, France
  }
}

\resume{
Nous pr\'esentons un nouveau mod\`ele de textures obtenues comme r\'ealisations de champs appel\'es champs browniens fractionnaires pond\'er\'ement tensoris\'es. Ils sont obtenus par une relaxation de la structure de produit tensoriel qui appara\^it dans la d\'efinition des draps browniens fractionnaires. 
Des propriétés statistiques de ces champs, telles que l'auto-similarit\'e, la stationnarit\'e des accroissements rectangulaires, sont obtenues. 
Une g\'en\'eralisation \`a  des champs \`a auto-similarit\'e matricielle est propos\'ee et des simulations bas\'ees sur la repr\'esentation spectrale sont effectu\'ees. 
}
\abstract{
We present a new model of textures, obtained as realizations of a new class of fractional Brownian fields. These fields, called {\it weighted tensorized fractional Brownian fields}, are obtained by a relaxation of the tensor-product structure that appears in the definition of fractional Brownian sheets. Statistical properties such as self-similarity and stationarity of rectangular increments are obtained. An operator scaling extension is defined and we provide simulations of the fields using their spectral representation.
}

\maketitle

\section{Introduction}
 The modeling of natural phenomena, in particular textures, by random objects has led to the introduction of numerous stochastic processes and fields. 
The most famous and historically first example is the well-known Brownian motion, which was extended to fractional Brownian motions by Kolmogorov in his seminal  paper of 1940 \cite{MR0003441}, to define ``Gaussian spirals'' in Hilbert spaces. 
Given a Hurst parameter $H \in (0,1)$, the fractional Brownian motion $B^H$ is the unique Gaussian process with stationary increments satisfying the self-similarity relation $B^H_{at} \stackrel{(d)}{=} a^H B_t$ for any $a, t>0$, where $\stackrel{(d)}{=}$ means that the equality holds in the sense of finite-dimensional distributions. It can be characterized via its harmonizable representation:
\begin{equation}
\forall t \in \mathbb{R}_+, \, B_t^H= \int_{\mathbb{R}} \frac{e^{it\xi}-1}{\vert \xi \vert^{H+1/2}} d\hat{\bf{W}}(\xi), 
\end{equation}

Several extensions have been proposed in higher dimensions. In particular, two natural generalizations have been introduced. The first one is the Levy fractional Brownian motion (LFBM) of Hurst index $H \in (0,1)$, also called fractional Brownian field (see e.g. \cite{ST94}).  It is the unique real-valued isotropic Gaussian field  $Y^H$ with stationary increments satisfying the self-similarity property $Y^H_{a\bf{x}}\stackrel{(d)}{=} a^H Y^H_{\bf{x}}$, where ``isotropic'' means that the field is invariant in law by rotation. Again, it can be defined using its harmonizable representation:
\begin{equation}
Y^H_{\bf{x}}= \int_{\mathbb{R}^N} \frac{e^{i\langle \bf{x} , \boldsymbol{\xi}\rangle} -1}{ \|\boldsymbol{\xi} \|^{H+\frac{N}{2}}}  d\hat{\bf{W}}({\boldsymbol{\xi}}) 
\end{equation}
where $\langle \cdot, \cdot \rangle$ denotes the standard scalar product in $\mathbb{R}^N$. 

A second famous extension is given by the fractional Brownian sheet (fBs) studied in \cite{Kam96,ALP}. For a given vector ${\bf{H}}= (H_1,...H_N) \in (0,1)^N $, the fBs of Hurst index ${\bf{H}}$ is a real-valued centered Gaussian random field $S^{\bf{H}}$ with  the following harmonizable representation
\begin{equation}
S^{\bf{H}}_{\bf{x}} = \int_{\mathbb{R}^N} \prod_{m=1}^N \frac{ e^{i  x_m \xi_m }-1}{\vert \xi_m\vert^{H_m +\frac{1}{2}} } d\hat{\bf{W}}({\boldsymbol{\xi}}).
\end{equation}
Setting $H_m=\frac{1}{2}$ for each $m \in \{1, \dots, N\}$ yields the standard Brownian sheet. While LFBMs are isotropic, fBs exhibit a strong ``tensor-product'' structure even when $H_m=H$ for all $m \in \{1, \dots, N\}$ and no longer have stationary increments but instead possess rectangular stationary increments (see Definition \ref{def:rectangular}). Nevertheless, this field has been widely studied for its interesting mathematical aspects, including fractal dimensions \cite{MR2169474}, geometric features, local times, and many more.

{\bf Further developments.} Focusing on the class of Gaussian fields, the two previously mentioned extensions of the fractional Brownian motions can be seen as particular cases of more general models. Important properties have been introduced in these models, such as anisotropy with different properties along directions. 

The Operator Scaling Gaussian Random Fields (OSGRF) introduced in \cite{BMS07} satisfy a self-similarity condition
$$
\forall a>0, \quad Z_{a^E {\bf{x}}} \stackrel{(d)}{=} a^H Z_{\bf{x}}
$$
for some $H>0$, where $E$ is a $N \times N$ matrix with eigenvalues having positive real parts, and where
$\displaystyle{
a^E = \sum_{k \ge 0} \frac{\ln (a)^k E^k}{k!}.}
$
These fields have been shown to exhibit anisotropic regularity properties \cite{BL09}, which offer strategies for numerical estimations of the parameters of the model \cite{RCVJA13}.

{\bf Goals, contribution and outline. } The strong effect of the tensor-product in Brownian sheets make them insufficiently realistic  for modeling many  textures. However, in urban data or medical images, some textures may exhibit  a reticulated structure. The contribution of the paper is to provide new Gaussian textures with a controlled tensor-product effect. This effect emerges as a parameter $\alpha$ goes from $1$ to $0$, yielding the fractional Brownian sheet for $\alpha=0$ and a field closer to an LFBM in terms of regularity for $\alpha=1$. 
Note that the regularity of the fields has been studied in detail in \cite{ELV24}, where associated ``weighted tensorized'' function spaces have also been defined. This notion of weighted tensorized regularity naturally arises in partial differential equations such as the electronic Schr\"odinger equation \cite{Yse1}.
In this paper, we focus on statistical properties of self-similarity, stationarity and variance of increments. We provide an extension to the anisotropic cases and simulations of these fields.

\section{Definition of the fields}

Let $\alpha \in [0,1]$ and $H \in (0,1)$, we set
\[ H_\alpha^+ := (1+\alpha)H \text{ and } H_\alpha^- := (1-\alpha)H\]
and we define a weighted tensorized fractional Brownian field (WTFBF) as the Gaussian field $\{X^{\alpha,H}_{(x_1,x_2)}\}_{(x_1,x_2) \in \mathbb{R}^2}$ given by 
\begin{equation}\label{eqn:defchamp}
X^{\alpha,H}_{(x_1,x_2)} :=  \int_{\mathbb{R}^2}  \frac{(e^{i x_1 \xi_1}-1) (e^{i x_2 \xi_2}-1)}{\phi_{\alpha,H}(\xi_1,\xi_2)}  d\hat{\bf{W}}({\boldsymbol{\xi}}) 
\end{equation}
where the function
\[\phi_{\alpha,H}(\xi_1,\xi_2)=\min(\vert \xi_1 \vert, \vert \xi_2 \vert)^{H_\alpha^- +\frac{1}{2}}\max (\vert \xi_1 \vert, \vert \xi_2 \vert)^{H_\alpha^++\frac{1}{2} }\]
denotes the square root of the inverse of the spectral density of the field. In the sequel, we also use the notation
\[ \mathcal{K}^{\alpha,H}_{(x_1,x_2)}(\xi_1,\xi_2) := \frac{(e^{i x_1 \xi_1}-1) (e^{i x_2 \xi_2}-1)}{\phi_{\alpha,H}(\xi_1,\xi_2)} \]
for the kernel in the stochastic integral \eqref{eqn:defchamp}. Note that the field \eqref{eqn:defchamp} is well-defined since this last kernel belongs to $L^2(\mathbb{R}^2)$. Note furthermore that the Fourier transform of $ \mathcal{K}^{\alpha,H}_{(x_1,x_2)}$ is real. Indeed, one has 
\begin{align*}
   &  \Im \left( e^{-i(t_1\xi_1 + t_2 \xi_2 )} (e^{i x_1 \xi_1}-1) (e^{i x_2 \xi_2}-1)\right)
   \\
    =&  \sin ((x_1-t_1)\xi_1 + (x_2-t_2) \xi_2) - \sin(-t_1 \xi_1 + (x_2-t_2) \xi_2) \\
   & - \sin ((x_1-t_1)\xi_1-t_2 \xi_2) + \sin(-t_1\xi_1 - t_2\xi_2)
\end{align*}
which is an odd function in $(\xi_1,\xi_2)$. It follows that 
$$\int_{\mathbb{R}^2}  \frac{ \Im \left( e^{-i(t_1\xi_1 + t_2 \xi_2 )}(e^{i x_1 \xi_1}-1) (e^{i x_2 \xi_2}-1)\right)}{\phi_{\alpha,H}(\xi_1,\xi_2)}  d\boldsymbol{\xi} = 0. 
$$
It implies that $\widehat{ \mathcal{K}^{\alpha,H}_{(x_1,x_2)}}$ is real, hence so  is  the field $X^{\alpha,H}_{(x_1,x_2)}$.

\section{Basic properties}

\begin{Prop}\label{prop:selfsimilar}
For all $\alpha \in [0,1]$ and $H \in (0,1)$, the process $X^{\alpha,H}$ is self-similar:  for all $a>0$, $\{X^{\alpha,H}_{(a x_1, a x_2)}\}_{(x_1,x_2)\in \mathbb{R}^2}\stackrel{(d)}{=}\{a^{2H} X^{\alpha,H}_{(x_1, x_2)}\}_{(x_1,x_2)\in \mathbb{R}^2}$.
\end{Prop}
\begin{proof}
For all $(x_1,x_2) \in \mathbb{R}^2$ and $a>0$, we have
\begin{align*}
 X^{\alpha,H}_{(a x_1, a x_2)} 
&=  \int_{\mathbb{R}^2}   \frac{(e^{i a x_1 \xi_1}-1) (e^{i a x_2 \xi_2}-1)}{\phi_{\alpha,H}(\xi_1,\xi_2)}  d\hat{\bf{W}}({\boldsymbol{\xi}}) \\
&\stackrel{(d)}{=} \int_{\mathbb{R}^2}   \frac{(e^{i x_1 \eta_1}-1) (e^{i  x_2 \eta_2}-1)}{\phi_{\alpha,H}(\frac{\eta_1}{a},\frac{\eta_2}{a})} a^{-1} d\hat{\bf{W}}({\boldsymbol{\eta}})  \\
&=a^{2H } \int_{\mathbb{R}^2}   \frac{(e^{i x_1 \eta_1}-1) (e^{i  x_2 \eta_2}-1)}{\phi_{\alpha,H}(\eta_1,\eta_2)}  d\hat{\bf{W}}({\boldsymbol{\eta}}) \\
&=a^{2H } X^{\alpha,H}_{( x_1,  x_2)},
\end{align*}
where we used the change of variables $(\eta_1,\eta_2)=(a\xi_1,a\xi_2)$ in the stochastic integral.
\end{proof}

Classically, the stationarity of increments is a too strong property for stochastic fields, and it is preferable to use the property of stationarity for rectangular increments \cite{ALP,makomishu}.
\begin{Def}\label{def:rectangular}
If $\{X_{(x_1,x_2)}\}_{(x_1,x_2) \in \mathbb{R}^2}$ is a field and  if $(x_1,x_2),(y_1,y_2) \in \mathbb{R}^2$, we set
\begin{align*}
    &\Delta  X_{(x_1,x_2);(y_1,y_2)} \\
    &:= X_{(x_1+y_1,x_2+y_2)}-X_{(y_1,x_2+y_2)}-X_{(x_1+y_1,y_2)}+X_{(y_1,y_2)}.
\end{align*}
We say that $\{X_{(x_1,x_2)}\}_{(x_1,x_2) \in \mathbb{R}^2}$ has stationary rectangular increments if, for any $(y_1,y_2) \in \mathbb{R}^2$, we have
\begin{align*}
\{\Delta X_{(x_1,x_2);(y_1,y_2)} \}_{(x_1,x_2) \in \mathbb{R}^2}  \stackrel{(d)}{=}\{X_{(x_1,x_2)}\}_{(x_1,x_2) \in \mathbb{R}^2}.
\end{align*}
\end{Def}

\begin{Prop}\label{prop:stat}
    For all {$\alpha \in [0,1]$ and $H \in (0,1)$}, the field $\{X^{\alpha,H}_{(x_1,x_2)}\}_{(x_1,x_2) \in \mathbb{R}^2}$ has stationary rectangular increments.
\end{Prop}

\begin{proof}
    First, we remark that for any $(x_1,x_2),(y_1,y_2) \in \mathbb{R}^2$, we get from \eqref{eqn:defchamp}
    \begin{align*}
\Delta  X_{(x_1,x_2);(y_1,y_2)}^{\alpha,H} \! =  \!\int_{\mathbb{R}^2} e^{i(y_1 \xi_1 + y_2 \xi_2)}\mathcal{K}^{\alpha,H}_{(x_1,x_2)}(\xi_1,\xi_2)  d\hat{\bf{W}}({\boldsymbol{\xi}}) .
    \end{align*}
Thus, recalling \cite[Corollary 6.3.2]{ST94}, we have 
\begin{align*}
   & \mathbb{E} \left( \exp\left(i \sum_{j=1}^n t^{(j)} \Delta  X_{(x_1^{(j)},x_2^{(j)});(y_1,y_2)}^{\alpha,H} \right)\right) \\
    &\! =\exp \!\! \left( \!\!-c_0\!\! \int_{\mathbb{R}^2} \! \left\vert \sum_{j=1}^n t^{(j)}  e^{i(y_1 \xi_1 + y_2 \xi_2)}\mathcal{K}^{\alpha,H}_{(x_1^{(j)},x_2^{(j)})}(\xi_1,\xi_2)\right\vert^2 \! \!\!\!d{\boldsymbol{\xi}}  \! \right) \\
    &\! = \mathbb{E} \left( \exp\left(i \sum_{j=1}^n t^{(j)}  X^{\alpha,H}_{(x_1^{(j)},x_2^{(j)})} \right) \right),
\end{align*}
for any $(x_1^{(1)},x_2^{(1)}),\dots,(x_1^{(n)},x_2^{(n)}),(y_1,y_2) \in \mathbb{R}^2$ and any $t^{(1)},\dots,t^{(n)} \in \mathbb{R}$, with 
\[ c_0 := \frac{1}{2\pi} \int_0^\pi \cos(\theta)^2 \, d\theta.\]
The conclusion follows directly.
\end{proof}

\setlength{\tabcolsep}{1pt}
\begin{figure*}
    \centering
    \begin{tabular}{cccccc}
       \includegraphics[width=0.162\textwidth]{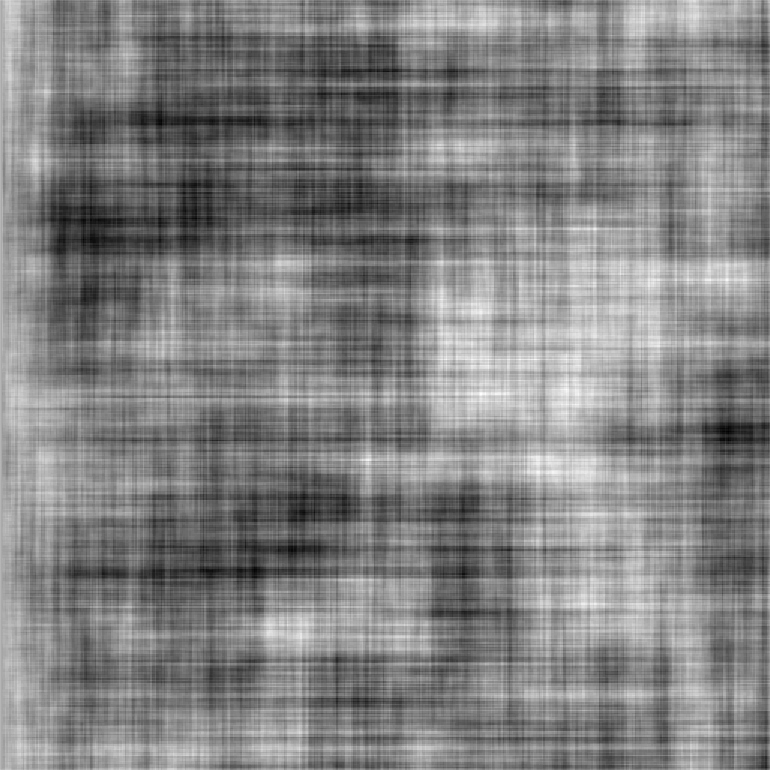}  & \includegraphics[width=0.162\textwidth]{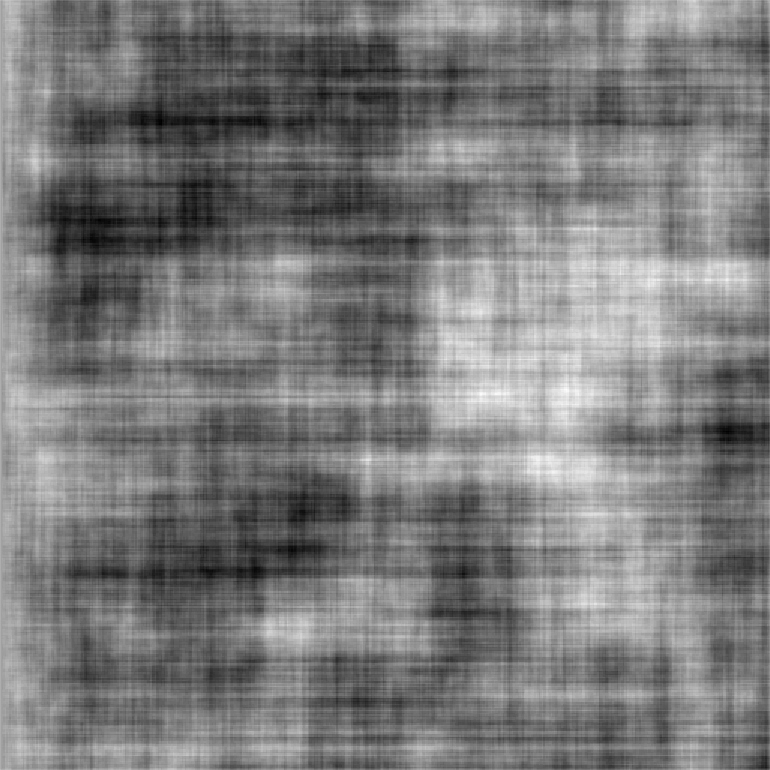} & \includegraphics[width=0.162\textwidth]{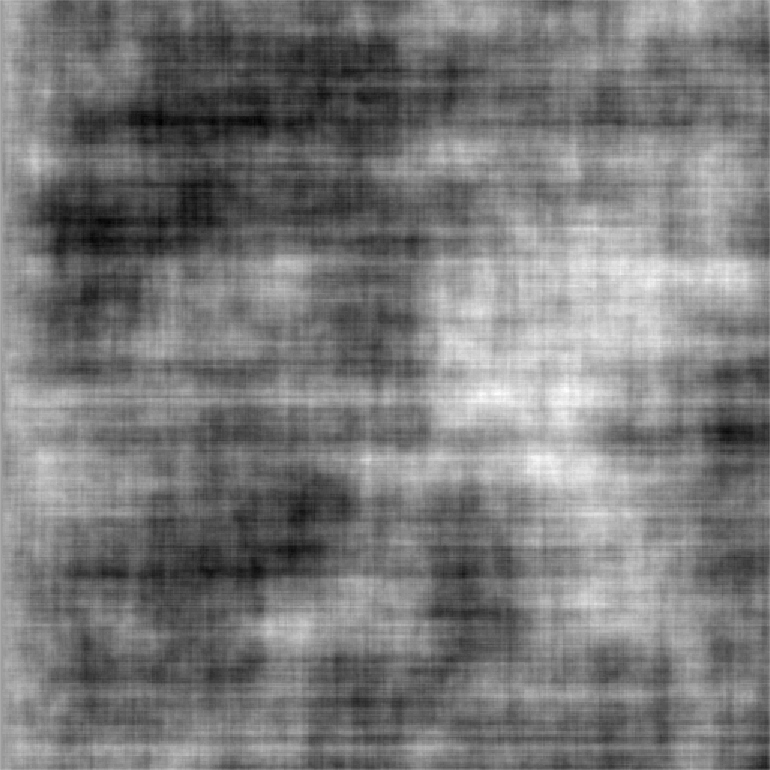} & \includegraphics[width=0.162\textwidth]{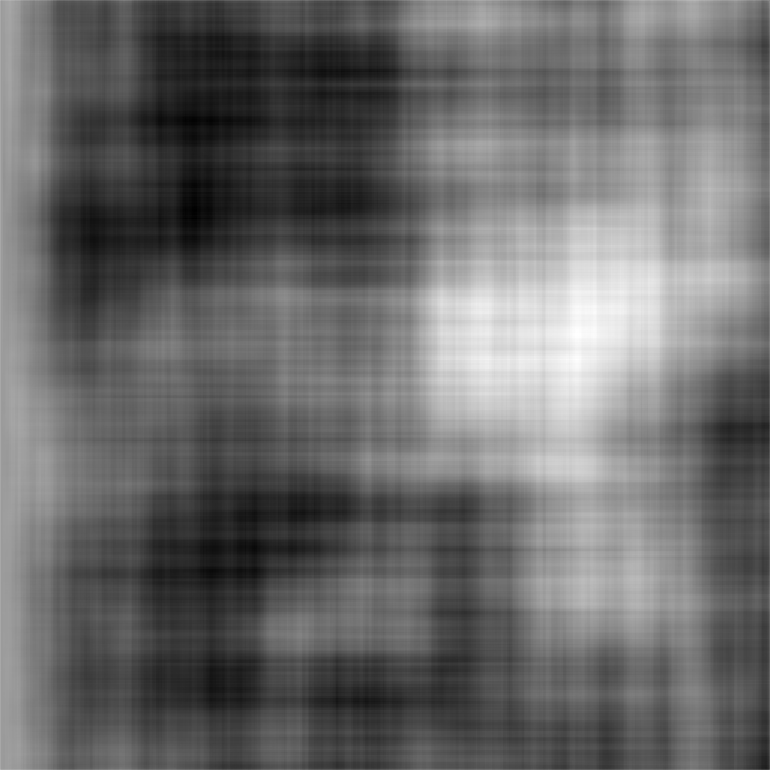} & \includegraphics[width=0.162\textwidth]{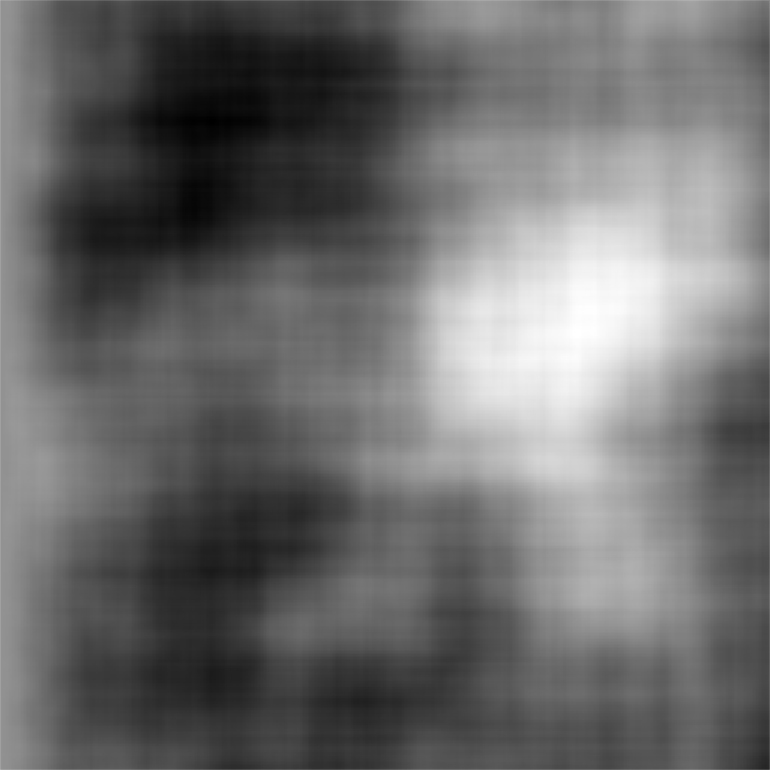} & \includegraphics[width=0.162\textwidth]{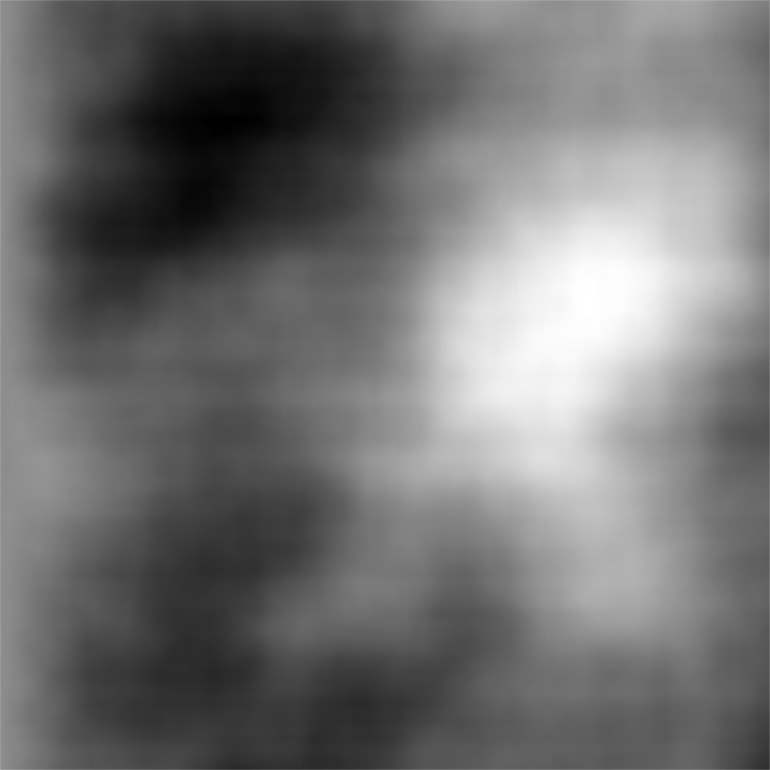} \\
         (a) $\alpha = 0$ & (b) $\alpha = 0.5$ & (c) $\alpha = 1$ & (d) $\alpha = 0$ & (e) $\alpha = 0.5$ & (f) $\alpha = 1$
    \end{tabular}    
    \caption{Weighted tensorized fractional Brownian fields simulated using a spectral representation approximation method, with parameters (a-c) $H~=~0.3$ or (d-f) $H~=~0.7$ and (a,d) $\alpha~=~0$, (b,e) $\alpha~=~0.5$ or (c,f) $\alpha~=~1$.}
    \label{fig:WTFBT_H03_H07_simu}
\end{figure*}

\section{Variance of rectangular increments}

\begin{Prop}\label{prop:variancerect}
For all {$\alpha \in [0,1]$ and $H \in (0,1)$}, there is a constant $c_1>0$ such that the rectangular increments of $\{X^{\alpha,H}_{(x_1,x_2)}\}_{(x_1,x_2) \in \mathbb{R}^2}$ satisfy 
\begin{align*}
   & \mathbb{E}(|\Delta  X^{\alpha,H}_{(h_1,h_2);(x_1,x_2)}|^2) \\
 & \,\,  \leq   c_1 \left(\max\{|h_1|,|h_2| \}^{1-\alpha}\min\{|h_1|,|h_2| \}^{1+\alpha}\right)^{2H}
\end{align*}
for all $(x_1,x_2),(h_1,h_2) \in \mathbb{R}^2$. 
\end{Prop}

\begin{proof}

The isometry property of the stochastic integral  gives  
\begin{align*} 
&\mathbb{E}\big(|\Delta  X^{\alpha,H}_{(h_1,h_2);(x_1,x_2)}|^2\big)\\
& = \int_{\mathbb{R}^2} \frac{|e^{i(x_1+h_1)\xi_1} - e^{ix_1\xi_1}|^2 |e^{i(x_2+h_2)\xi_2} - e^{ix_2\xi_2}|^2}{(\phi_{\alpha,H}(\xi_1,\xi_2))^2}  d{\boldsymbol{\xi}}\\
& = \frac{1}{|h_1| \, |h_2|}\int_{\mathbb{R}^2}  \frac{|e^{i\eta_1} -1|^2 |e^{i\eta_2} - 1|^2}{(\phi_{\alpha,H}(\frac{\eta_1}{h_1},\frac{\eta_2}{h_2}))^2}  d{\boldsymbol{\eta}}
\end{align*}
using the change of variables $(\eta_1,\eta_2) = (h_1 \xi_1, h_2\xi_2)$. Notice now that
if $|h_1|\geq|h_2|$, one has
\begin{align*}
(\phi_{\alpha,H}(\tfrac{\eta_1}{h_1},\tfrac{\eta_2}{h_2}))^2
& \geq \frac{\min(\vert \eta_1 \vert, \vert \eta_2  \vert)^{2H_\alpha^- +1}  \vert \eta_2 \vert^{2H_\alpha^++1 }}{|h_1|^{(2H_\alpha^- +1)}|h_2|^{(2H_\alpha^+ +1)} }.
\end{align*}
It implies $\displaystyle{
\mathbb{E}\big(|\Delta  X^{\alpha,H}_{(h_1,h_2);(x_1,x_2)}|^2\big) \leq c_1  |h_1|^{2H_\alpha^- }|h_2|^{2H_\alpha^+ }  }
$
with
\begin{equation}\label{eq:Constante}
c_1 =   \int_{\mathbb{R}^2}   \frac{|e^{i\eta_1} -1|^2 |e^{i\eta_2} - 1|^2}{\min(\vert \eta_1 \vert, \vert \eta_2  \vert)^{2H_\alpha^- +1}  \vert \eta_2 \vert^{2H_\alpha^++1 }}d{\boldsymbol{\eta}}    
\end{equation}
if $|h_1|\geq|h_2|$. The same argument for $|h_1| < |h_2|$ leads to the conclusion. 
\end{proof}

Regarding the rectangular increments, a generalization of Kolmogorov's continuity theorem allows us to assert that there exists a modification of the field $\{X^{\alpha,H}_{(x_1,x_2)}\}_{(x_1,x_2) \in \mathbb{R}^2}$ which is nearly \emph{locally $(H^+_\alpha, H^-_\alpha)$-rectangular H\"older}. This means that for every bounded intervals $I,J$ of $\mathbb{R}$, every $x_1 \in I$, $x_2 \in J$ and every $\varepsilon>0$, there exists a positive finite random variable $C>0$ such that almost surely
\begin{align*}
&|\Delta  X^{\alpha,H}_{(h_1,h_2);(x_1,x_2)}|\\
& \leq C \left( \max\{|h_1|,|h_2| \}^{(1-\alpha)}\min\{|h_1|,|h_2| \}^{(1+ \alpha)} \right)^{H-\varepsilon}
\end{align*}
for all $h_1, h_2 \in \mathbb{R}$ such that $x_1+h_1 \in I$ and $x_2 + h_2 \in J$, see \cite{ELV24} for details.

\section{Anisotropic extension}

The model introduced in the previous sections can be extended to provide anisotropic textures by imposing an operator scaling property.

We consider $\beta_1, \beta_2 \in (0,2) $ such that $\beta_1+\beta_2 =2$, and set
\begin{align}\label{eqn:OpScalingExtension}
X^{\alpha,H,\beta_1,\beta_2}_{(x_1,x_2)} :=  \int_{\mathbb{R}^2}  \frac{(e^{i x_1 \xi_1}-1) (e^{i x_2 \xi_2}-1)}{\phi_{\alpha,H,\beta_1,\beta_2}(\xi_1,\xi_2)}  d\hat{\bf{W}}({\boldsymbol{\xi}}) 
\end{align}
where $\phi_{\alpha,H,\beta_1,\beta_2}(\xi_1,\xi_2)=\phi_{\alpha,H}\big(|\xi_1|^\frac{1}{\beta_1},|\xi_2|^\frac{1}{\beta_2}\big).$
If  $\max(\beta_1,\beta_2)-1<2H < 3\min(\beta_1,\beta_2)-1$, the corresponding field is well-defined and satisfies
\begin{equation}
X^{\alpha, H, \beta_1,\beta_2}_{a^D \bf{x} }\stackrel{(d)}{=} a^{2H} X^{\alpha, H, \beta_1,\beta_2}_{\bf x}
\end{equation}
with $D=\mbox{diag}(\beta_1,\beta_2)$ and $a^D {\bf{x}} = (a^{\beta_1} x_1, a^{\beta_2} x_2)$.  These fields then satisfy anisotropic properties of regularities that will be explored in a forthcoming work.

\begin{figure*}
    \centering
    \begin{tabular}{cccccc}
\includegraphics[width=0.162\textwidth]{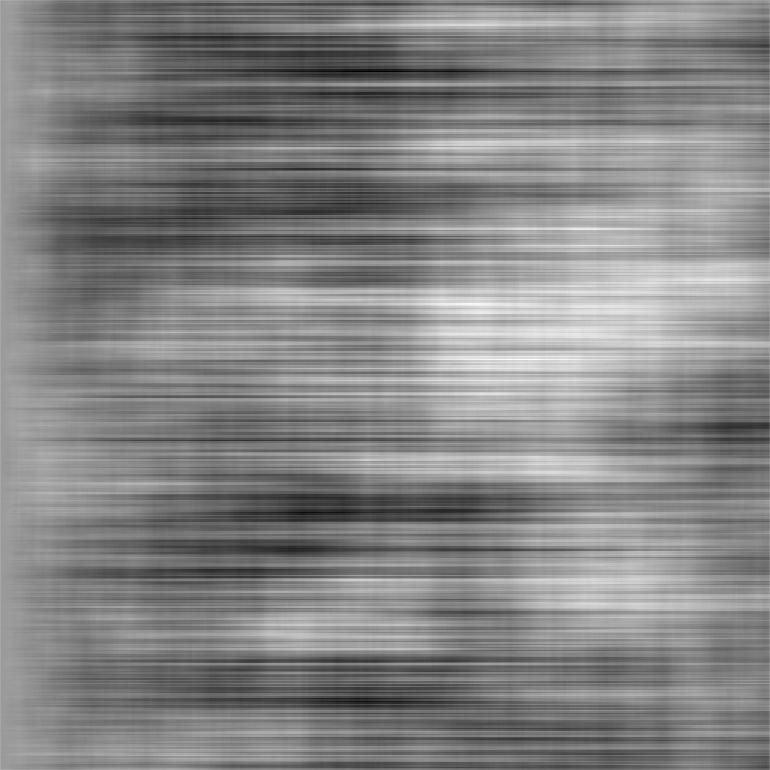} & \includegraphics[width=0.162\textwidth]{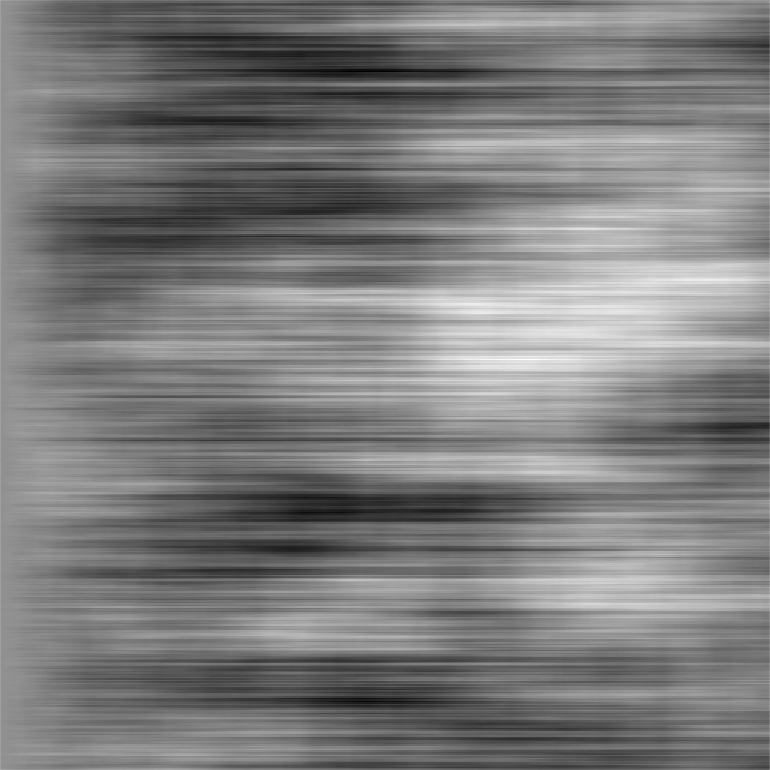} & \includegraphics[width=0.162\textwidth]{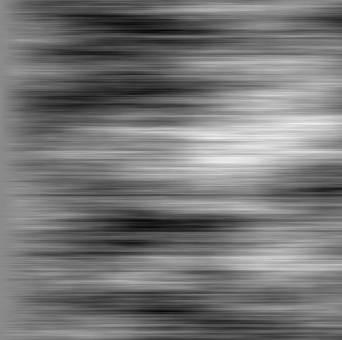} & \includegraphics[width=0.162\textwidth]{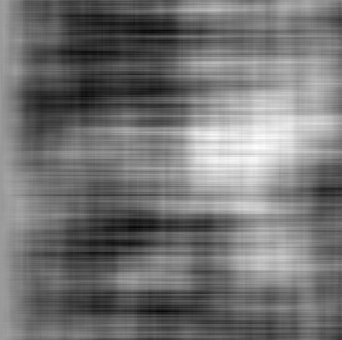} & \includegraphics[width=0.162\textwidth]{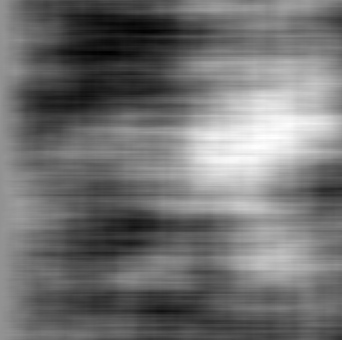}& \includegraphics[width=0.162\textwidth]{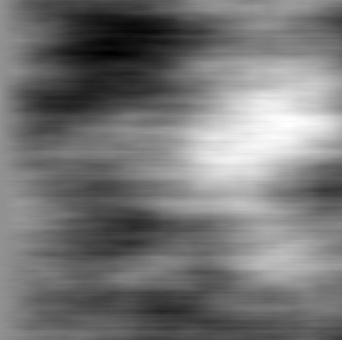} \\
         (a) $\alpha = 0$ & (b) $\alpha = 0.5$ & (c) $\alpha = 1$ & (d) $\alpha = 0$  & (e) $\alpha = 0.5$  & (f) $\alpha = 1$
    \end{tabular}    
    \caption{Anisotropic weighted tensorized fractional Brownian fields simulated using a spectral representation approximation method, with parameters (a-c) $H=0.4$, $\beta_1=0.7$, $\beta_2=1.3$ and $\alpha = 0$, $\alpha = 0.5$ or $\alpha = 1$, and (d-f) $H=0.6$, $\beta_1=0.85$, $\beta_2=1.15$ and $\alpha = 0$, $\alpha = 0.5$ or $\alpha = 1$. }
    \label{fig:WTFBT_aniso1_simu}
\end{figure*}

\section{Simulation}
Several strategies have been developed to simulate Gaussian random fields. Methods based on an explicit expression for the covariance of the field allow for exact simulations that preserve statistical properties such as stationarity (see \cite{BL20}). When the covariance is not explicitly known but is known along radial directions, the turning-bands method can be employed, as in \cite{BMR15} to simulate some anisotropic fields. Using the spectral density of the field, approximations of AFBF have been obtained in \cite{ABE09}. Given that we only have an integral expression of the covariance, we adopt this approach to generate a WTFBF. This spectral representation method involves discretizing the field in the Fourier domain \cite{SD96}. Although its main limitations include challenges in assessing the convergence of the approximation, as well as the potential for the inverse Fourier transform to disrupt the statistical properties of the field, this method is still widely used to generate stationary and non-stationary Gaussian random fields. In the case of stationary random fields, Shinozuka and Deodatis \cite{SD96} proved that the generated samples verify ergodic properties, in the sense that the spatial mean and autocorrelation function of any sample converge to the theoretical mean and autocorrelation function of the field, as the sample size increases. It is also fast and easy to perform as it involves fast Fourier transforms. Approximations based on wavelet methods could be used, but they are known to be quite slow in practice even if they provide the best approximation rate by a series in the case of FBF.

The results presented in Figures \ref{fig:WTFBT_H03_H07_simu} and \ref{fig:WTFBT_aniso1_simu} are generated using a spectral representation approximation on a discrete grid of size $(M+1) \times (M+1)$, with $M=512$. For a given WTFBF $\{ X_{(x_1,x_2)}^{\alpha, H} \}_{(x_1,x_2) \in \mathbb{R}^2}$, the strategy involves generating $W$, a collection of independent standard complex Gaussian variables of size $(2M \times 2M)$. These variables are then multiplied by a function $g$. Next, in both directions successively, a 1D Fourier transform is applied, followed by subtracting the value of the field at the origin. If  we set  $
g(x,y) = 
    (\phi_{\alpha, H}(x,y))^{-1} \textbf{1}_{\{x \neq 0, y \neq 0\}}
$ for $(x,y) \in \mathbb{R}^2$,  the generated field $x^{\alpha,H}$ is given, for all $k_1, k_2 \in \{0,\dots, M\}$, by
$${ x^{\alpha,H} \left( \tfrac{k_1}{M}, \tfrac{k_2}{M} \right) = \mathcal{R}\left( y_2 \left( \tfrac{k_1}{M}, \tfrac{k_2}{M} \right) -  y_2\left(0,\tfrac{k_2}{M}\right) \right)},$$
where for any $n_1 \in \{ - M+1, \dots, M\}$
$$\displaystyle{ y_1 \left( n_1, \tfrac{k_2}{M} \right) = \hspace{-3mm}\sum_{n_2=-M+1}^M \hspace{-3mm} W(n_1,n_2) g\left( \pi n_1, \pi n_2 \right) e^{- \frac{2i \pi n_2 k_2}{2M}}},$$
$$\displaystyle{y_2 \left( \tfrac{k_1}{M}, \tfrac{k_2}{M} \right) = \pi \hspace{-3mm} \sum_{n_1=-M+1}^M \hspace{-3mm} \left( y_1 \left( n_1, \tfrac{k_2}{M} \right) - y_1(n_1, 0) \right) e^{- \frac{2i \pi n_1 k_1}{2M}}}.$$

The same method is used to simulate the anisotropic extension. 
Figure \ref{fig:WTFBT_H03_H07_simu} presents synthesized  WTFBFs  with various parameters $H$ and $\alpha$. When $\alpha =0$, the procedure samples a fBs while, when $\alpha=1$, the generated texture tends to loose the reticulated aspect and to approach a fractional Brownian field.
Figure \ref{fig:WTFBT_aniso1_simu} shows  anisotropic WTFBFs.  These fields  $X^{\alpha,H,\beta_1,\beta_2}$ produce anisotropic textures, where the highest $\beta$ determines the dominant direction. The images (a-f) illustrate the effects of the parameters $\alpha$, $\beta$ and $H$ on the fields. 
We generated 100 textures of size $512 \times 512$, with parameters $H =0.3$ and $\alpha = 0.5$ to evaluate the statistical properties of the realizations, similarly as in \cite{VS2021}. The results are presented in Table \ref{table_simu}. Note that the fields given by the rectangular increments tend to have a mean and a skewness approaching the mean and skewness of the original field and the theoretical mean and skewness, supposed to be zero. Regarding the self-similarity property, the estimated mean and skewness of the rescaled generated textures are close to their theoretical target, zero. In both cases however, the variance seems to be biased. A matlab implementation to generate these textures and reproduce our results is available online (https://github.com/claunay/wtfbf).

\vspace{-0.5cm}
\begin{table}[htbp]
\caption{Estimated moments of WTFBFs generated by the spectral representation method}
\vspace{-3mm}
\begin{center}
\begin{tabular}{|c|c|c|c|}
\hline
& $X^{\alpha,H}$ & $\Delta X^{\alpha,H}$ &$\frac{1}{a^{2H}}X^{\alpha,H}$\\
\hline
Mean  & $-2 \times 10^{-4}$& $-1 \times 10^{-5}$& $2 \times 10^{-2}$\\
\hline
Variance  & 7.3 & 10.7 & 1.4\\
\hline
Skewness  & $6 \times 10^{-4}$ & $1 \times 10^{-6}$& -0.3\\
\hline
\end{tabular}
\end{center}
\label{table_simu}
\end{table}

\vspace{-1.1cm}
\section*{Conclusion}

In this paper, we introduced a new model of textures that relaxes the tensor-product structure of fractional Brownian sheets, in an attempt to construct a bridge between fBs and fractional Brownian fields. These textures can appear in medical or urban images and are associated with regularity properties involved partial differential equations. We have shown that these random fields are self-similar and have stationary rectangular increments, with bounded variance. Some simulations illustrate the behavior of these textures for various parameters.

\vspace{2mm}

{\bf{Acknowledgment}}
This work was conducted within  the ANR Mistic project (ANR-19-CE40-005).

\vspace{-4mm}

\bibliography{biblio}

\end{document}